\documentclass[12pt]{article}
\usepackage[margin=1in]{geometry}
\usepackage{amsthm, amsmath,amsfonts,amssymb,euscript,hyperref,graphics,color,slashed,mathrsfs}
\usepackage{graphicx}
\usepackage{comment}
\usepackage{import}
\usepackage{tikz}
\usepackage{latexsym}
\usepackage{mathtools}
\usepackage{appendix}
\usepackage{stmaryrd}
\usetikzlibrary{arrows}
\usepackage{pgfplots}
\usepackage{extarrows}
\usepackage{parskip}

\allowdisplaybreaks

\newtheorem{theorem}{Theorem}[section]
\newtheorem*{theorem*}{Theorem}

\newtheorem{lemma}[theorem]{Lemma}

\newtheorem{remark}[theorem]{Remark}

\numberwithin{equation}{section}

\begin{document}
\title {A new elementary proof of the formula $\sum\limits_{n=1}^{\infty}\frac{1}{n^2}=\frac{\pi^2}{6}$}

\author{Jia Li}
\date{}
%\address{}
%\email{jialimath001@pku.org.cn}
%\subjclass[2020]{11J72 (primary), 11M06, 33C20 (secondary)}
%\keywords{}

\maketitle

%\begin{abstract}
%In this article, we provide a new elementary proof of the formula $\sum\limits_{n=1}^{\infty}\frac{1}{n^2}=\frac{\pi^2}{6}$.
%\end{abstract}

\section{Introduction}

In the history of mathematics, the Basel problem was initially resolved by Euler. To date, dozens of different proofs have been developed. The most elementary approach was presented by Papadimitriou in \cite{PI1973}. Later, Apostol \cite{AT1973} extended this method to evaluate $\zeta(2n)$. Although our approach is fundamentally based on that of Papadimitriou, we reformulate the key identities using tools from linear algebra, thereby offering a self-contained derivation.

\section{Some useful lemmas}

\begin{lemma}\label{A_n}
    Let 
$$A_n:=\left[\begin{matrix}
n&n-1&n-2&\cdots&2&1\\
n-1&n-1&n-2&\cdots&2&1\\
n-2&n-2&n-2&\cdots&2&1\\
\vdots&\vdots&\vdots&\ddots&\vdots&\vdots\\
2&2&2&\cdots&2&1\\
1&1&1&\cdots&1&1
\end{matrix}\right]$$
be a $n\times n$ matrix, then the eigenvalues of the matrix $A_n$ are
$$\lambda_k=\frac{1}{4\sin^2\left(\frac{2k-1}{2n+1}\cdot\frac{\pi}{2}\right)}, k=1,2,\cdots,n.$$
\end{lemma}

\begin{proof}
     Let $B_n:=A_n^{-1}$, a direct calculation shows that
$$B_n=\left[\begin{matrix}
1&-1&0&\cdots&0&0\\
-1&2&-1&\cdots&0&0\\
0&-1&2&\cdots&0&0\\
\vdots&\vdots&\vdots&\ddots&\vdots&\vdots\\
0&0&0&\cdots&2&-1\\
0&0&0&\cdots&-1&2
\end{matrix}\right]$$
we consider the following determinant
\begin{align*}
    D_n(\theta):&=\det(4\sin^2(\theta)I_n-B_n)\\
    &=\left|\begin{matrix}
4\sin^2(\theta)-1&1&0&\cdots&0&0\\
1&4\sin^2(\theta)-2&1&\cdots&0&0\\
0&1&4\sin^2(\theta)-2&\cdots&0&0\\
\vdots&\vdots&\vdots&\ddots&\vdots&\vdots\\
0&0&0&\cdots&4\sin^2(\theta)-2&1\\
0&0&0&\cdots&1&4\sin^2(\theta)-2
\end{matrix}\right|
\end{align*}
this leads to the following recurrence relation:
$$\begin{cases}
D_n(\theta)=(4\sin^2(\theta)-2)D_{n-1}(\theta)-D_{n-2}(\theta),n\geqslant 3\\ 
D_1(\theta)=1-2\cos(2\theta)\\
D_2(\theta)=1-2\cos(2\theta)+2\cos(4\theta)
\end{cases}$$
it follows by induction that
$$D_n(\theta)=1-2\cos(2\theta)+2\cos(2\theta)+\cdots+(-1)^n2\cos(2n\theta)=(-1)^n\frac{\cos((2n+1)\theta)}{\cos(\theta)}$$
hence, we obtain
$$D_n\left(\frac{2k-1}{2n+1}\cdot\frac{\pi}{2}\right)=0, k=1,2,\cdots,n.$$
Since $4\sin^2\left(\frac{2k-1}{2n+1}\cdot\frac{\pi}{2}\right)\neq4\sin^2\left(\frac{2l-1}{2n+1}\cdot\frac{\pi}{2}\right)(k\neq l)$, we obtain $n$ different eigenvalues of $B_n$. Hence, the eigenvalues of matrix $B_n$ are
$$4\sin^2\left(\frac{2k-1}{2n+1}\cdot\frac{\pi}{2}\right), k=1,2,\cdots,n.$$
Then, the eigenvalues of matrix $A_n$ are
$$\lambda_k=\frac{1}{4\sin^2\left(\frac{2k-1}{2n+1}\cdot\frac{\pi}{2}\right)}, k=1,2,\cdots,n.$$
\end{proof}

\begin{lemma}\label{sumlambda}
    We have the following identity
$$\sum_{k=1}^n\cot^2\left(\frac{2k-1}{2n+1}\cdot\frac{\pi}{2}\right)=2n^2+n$$
\end{lemma} 

\begin{proof}
By lemma \ref{A_n}, we have
\begin{align*}
    \sum_{k=1}^n\cot^2\left(\frac{2k-1}{2n+1}\cdot\frac{\pi}{2}\right)=\sum_{k=1}^n\frac{1}{\sin^2\left(\frac{2k-1}{2n+1}\cdot\frac{\pi}{2}\right)}-n
    =4\cdot\text{tr}(A_n)-n=2n^2+n
\end{align*}
\end{proof}
\begin{remark}\label{cot4}
    The above method can also be used to obtain the following formula
    $$\sum_{k=1}^n\cot^4\left(\frac{2k-1}{2n+1}\cdot\frac{\pi}{2}\right)=\frac{8n^4+16n^3+4n^2-n}{3}$$
\end{remark}

\section{Proof of $\sum\limits_{n=1}^{\infty}\frac{1}{n^2}=\frac{\pi^2}{6}$}
We have well-known inequalities
$$\sin(\theta)<\theta<\tan(\theta), 0<\theta<\frac{\pi}{2}$$
hence, we get the inequality
$$\cot^2(\theta)<\frac{1}{\theta^2}<\cot^2(\theta)+1, 0<\theta<\frac{\pi}{2}$$
then, we have
$$\sum_{k=1}^n\cot^2\left(\frac{2k-1}{2n+1}\cdot\frac{\pi}{2}\right)<\frac{4}{\pi^2}\sum_{k=1}^n\frac{(2n+1)^2}{(2k-1)^2}<\sum_{k=1}^n\cot^2\left(\frac{2k-1}{2n+1}\cdot\frac{\pi}{2}\right)+n$$
by lemma \ref{sumlambda}, we obtain
$$\frac{\pi^2}{4}\cdot\frac{2n^2+n}{(2n+1)^2}<\sum_{k=1}^n\frac{1}{(2k-1)^2}<\frac{\pi^2}{4}\cdot\frac{2n^2+2n}{(2n+1)^2}\Longrightarrow\sum_{n=1}^{\infty}\frac{1}{(2n-1)^2}=\frac{\pi^2}{8}$$
hence, we have
$$\sum_{n=1}^{\infty}\frac{1}{n^2}=\frac{4}{3}\sum_{n=1}^{\infty}\frac{1}{(2n-1)^2}=\frac{4}{3}\cdot\frac{\pi^2}{8}=\frac{\pi^2}{6}.$$
Similarly, by using remark \ref{cot4}, we also have
$$\sum_{k=1}^{\infty}\frac{1}{n^4}=\frac{\pi^4}{90}.$$

\begin{flushright}
\begin{minipage}{148mm}\sc\footnotesize
J.\,L.: School of Mathematical Sciences, \\ Peking University, Beijing, China \\
{\it E-mail address}: \href{mailto:jialimath001@pku.org.cn}{{\tt jialimath001@pku.org.cn}} \vspace*{3mm}
\end{minipage}
\end{flushright}

\end{document}